\documentclass[a4paper]{amsart}

\usepackage{array}              
\usepackage{amssymb}              
\usepackage{lmodern,bm}              
\usepackage{longtable}
\usepackage{tikz}
\usepackage[all]{xy}
\usepackage{dsfont}
\usepackage{pdflscape}
\usepackage{bigdelim}
\usepackage{multirow}
\usepackage{wasysym}

\CompileMatrices

\newtheorem{theorem}{Theorem}[section]
\newtheorem{introthm}{Theorem}
\newtheorem{lemma}[theorem]{Lemma}
\newtheorem{proposition}[theorem]{Proposition}
\newtheorem{corollary}[theorem]{Corollary}

\theoremstyle{definition}

\newtheorem{definition}[theorem]{Definition}

\newtheorem{example}[theorem]{Example}
\newtheorem{remark}[theorem]{Remark}

\numberwithin{equation}{theorem}

\def\vector2#1#2{\left(\begin{array}{c} #1 \\ #2 \end{array}\right)}

\def\SL{{\rm SL}}

\def\CC{{\mathbb C}}

\def\ZZ{{\mathbb Z}}

\def\GR{{\rm Gr}}

\def\multiset#1#2{\left(\!\!\!\left(
\begin{matrix}
#1 \\
#2
\end{matrix}
\right)\!\!\!\right)}

\def\stir#1#2{\left\{
\begin{matrix}
#1 \\
#2
\end{matrix}
\right\}}

\def\smallstir#1#2{\left\{
\begin{smallmatrix}
#1 \\
#2
\end{smallmatrix}
\right\}}

\def\SL{{\rm SL}}

\title[On the Hilbert series of the Grassmannian]{On the Hilbert series of the Grassmannian}

\author[Lukas Braun]{Lukas Braun}

\address{Mathematisches Institut, Universit\"at T\"ubingen,
Auf der Morgenstelle 10, 72076 T\"ubingen, Germany}
\email{braun@math.uni-tuebingen.de}


\begin{document}

\begin{abstract}
We compute the Hilbert series of the complex Grassmannian using invariant theoretic methods and show that its $h$-polynomial coincides with the $k$-Narayana polynomial. We give a simplified formula for the $h$-polynomial of Schubert varieties. Finally, we use a generalized hypergeometric Euler transform to find simplified formulae for the $k$-Narayana numbers, i.e. the $h$-polynomial of the Grassmannian.

\end{abstract}

\maketitle

\section*{Introduction}

Consider the Grassmannian $X = \GR(k,n)$ of k-dimensional
vector subspaces of a given n-dimensional complex
vector space and its homogeneous coordinate ring
$R = \oplus_j R_j$ defined by the Pl\"ucker embedding.
Recall that the associated Hilbert series is
$$
H(X)
\ = \
\sum_{j \ge 0} \dim (R_j) t^j.
$$
The Hilbert function $j \mapsto \dim (R_j)$ is, up to
finitely many values, a polynomial in~$j$, the Hilbert
polynomial of $\GR(k,n)$.
Moreover, the Hilbert series is represented
as a rational function in $j$ with a denominator polynomial
of degree $k(k(n-k)+1)$. The numerator is then called the $h$-polynomial
of $\GR(k,n)$.
Various approaches leading to explicit formulae for the Hilbert polynomial of $\GR(k,n)$ can be found for example in~~\cite{Chip, GN, Ho1, Ho2, HP, Li, Mu, SSW, Stu}.

Mukai used an invariant theoretic approach in~\cite{Mu} to compute the Hilbert polynomial of the Grassmannian in the special case $k=2$.
The aim of the present paper is to generalize this invariant theoretic method to arbitrary $k$.
In fact,  Mukai's Hilbert polynomial for $\GR(2,n)$ is the special case $k=2$, $r=n-1$ of the order polynomial
$$
\mathfrak{N}_k(r,j):= \prod_{i=0}^{k-1}\binom{r+i+j-1}{i+j}\binom{r+i-1}{i}^{-1},
$$ 
which is a polynomial in $j$, showing up in Sulanke's papers~\cite{Su1,Su2}. In these papers, he generalized the (two-dimensional) \emph{Narayana numbers} to the \emph{$k$-Narayana numbers} 
$$
N_k(r,j):=
\sum\limits_{l=0}^{j}(-1)^{j-l} \binom{kr+1}{j-l} \mathfrak{N}_k(r-1,l)
$$
for arbitrary dimension $k$. We call the order polynomial $\mathfrak{N}_k(r,j)$ the \emph{multiset $k$-Narayana numbers} in the following.  For combinatorial interpretations of these numbers, we refer to Section~\ref{sec:nar}. The generating series
$$
N_{k,r}:=\sum_{j=0}^{(k-1)(r-1)} N_k(r,j) t^j,
\qquad
\mathfrak{N}_{k,r}:= \sum_{j\geq 0} \mathfrak{N}_k(r,j) t^j,
$$
of which the first is a polynomial in $t$, are called 
the \emph{$k$-Narayana polynomial} and the \emph{$k$-Narayana series} respectively.
Computing the Hilbert series of the Grassmannian using the invariant theoretic method, we arrive at our first main theorem:

\begin{introthm}[Thm.~\ref{th:hilbgrass}]
The $h$-polynomial of the Grassmannian $\GR(k,n)$ is the $k$-Narayana polynomial $N_{k,n-k+1}$ and its Hilbert series is the $k$-Narayana series $\mathfrak{N}_{k,n-k+1}$.
\end{introthm}

Finally, we express the $k$-Narayana series as a hypergeometric function ${}_{k}F_{k-1}$. This leads to the observation that the simplified formula $\frac{1}{r}\binom{r}{j}\binom{r}{j+1}$ for the $2$-Narayana numbers is a direct consequence of Euler's hypergeometric transformation
$$
{}_{2}F_{1}\left(
\begin{matrix}
a,b\\
c
\end{matrix};t\right) = (1-t)^{c-a-b} {}_{2}F_{1}\left(
\begin{matrix}
c-a,c-b\\
c
\end{matrix};t\right).
$$
This transformation has been generalized in~\cite{MP2} - see also~\cite{Ma,MP1} - to the generalized hypergeometric function ${}_{k}F_{k-1}$. Applying this generalized Euler transformation to the $k$-Narayana series, we express the $k$-Narayana polynomial as a hypergeometric function, i.e. we find a new (multiplicative) formula for the $k$-Narayana numbers:

\begin{introthm}[Thm.~\ref{th:multform}]
For the $k$-Narayana numbers $N_k(r,j)$, we have the product formula
$$
N_k(r,j)= \binom{(k+1)r+2k+3}{j}\binom{(k+1)r+2k+2}{j} \prod_{i=1}^{(k-2)(r-3)}\frac{\eta_i+j}{\eta_i},
$$
where the $\eta_i$ are the zeros of a certain polynomial of degree $(k-2)(r-3)$.
\end{introthm}

The paper is divided in three sections. In the first, we study different generalizations of Narayana numbers, relations among them and express their generating functions as hypergeometric functions. The second section is devoted to the computation of Hilbert series of the Grassmannian using invariant theoretic methods. In addition, a simplified formula - compared to that of~\cite{Nan} - for the $h$-polynomial of Schubert varieties is derived.
The third and last section uses the generalized hypergeometric Euler transformation of~\cite{MP2} to find a new multiplicative formula for the $k$-Narayana numbers.

The author wishes to thank J\"urgen Hausen for many helpful comments.

\tableofcontents

\section{The $k$-Narayana numbers}\label{sec:nar}

The \emph{Narayana numbers}
$$
N(r,j):=\frac{1}{r}\binom{r}{j}\binom{r}{j+1} 
=
\frac{1}{j+1}\binom{r-1}{j}\binom{r}{j}
$$
count the number of Dyck paths in the plane from $(0,0)$ to $(r,r)$ with exactly $j$ ascents. They have been generalised in numerous ways. We exhibit some of these generalizations.
The \emph{$k$-Narayana} numbers $N_k(r,j)$ count the number of paths along the lattice $\ZZ^k$ from the origin to $(r,\ldots,r)$, staying in the region $\{0 \leq x_1 \leq \ldots \leq x_k\}$ and having $j$ \emph{ascents}. An ascent here is a pair of successive steps so that the second step is an increase in a coordinate with higher index than the first one.
They have been observed by Sulanke in the papers~\cite{Su1} and \cite{Su2} and are given by the formula
$$
N_k(r,j):=
\sum\limits_{l=0}^{j}(-1)^{j-l} \binom{kr+1}{j-l} \prod\limits_{i=0}^{k-1}\binom{r+i+l}{r}\binom{r+i}{r}^{-1}.
$$

In his initial paper~\cite{Na}, Narayana introduced the Narayana numbers in another context than that of the Dyck paths that coincides for $k=2$. Consider the following setting: paths with $j$ steps from the origin to a point $(a_1,\ldots,a_k)$ in $\ZZ^k$ with $a_1\geq \ldots \geq a_k \geq j$. The steps comply with the following rules:
\begin{enumerate}
\item in each step, each coordinate increases at least by one.
\item if $a_{il}$ is the $i$-th coordinate after the $l$-th step, then $a_{1l} \geq \ldots \geq a_{kl}\geq l$ holds.
\end{enumerate} 

The number $(a_1,\ldots,a_k)_{j}$ of such paths according to Theorem 1 of~\cite{Na} is given by
$$
(a_1,\ldots,a_k)_{j} := 
\begin{vmatrix}
\binom{a_1-1}{j} & \cdots & \binom{a_k-1}{j+k-1} \\
\vdots & \ddots & \vdots \\
\binom{a_1-1}{j-k+1} & \cdots & \binom{a_k-1}{j}
\end{vmatrix}
= \det\left( \binom{a_l-1}{j+l-i}_{1\leq i,l\leq k}\right).
$$

We see that by setting $k=2$, $a_1=a_2=r$ one gets the ordinary Narayana numbers, while for $k=1$, one gets the binomial coefficients. 
We now consider two small modifications, namely the numbers
\begin{small}

\begin{align*}
[a_1,\ldots,a_k]_{j} 
:=&
\begin{vmatrix}
\binom{a_1+j-1}{j} & \cdots & \binom{a_k+j-1}{j+k-1} \\
\vdots & \ddots & \vdots \\
\binom{a_1+j-1}{j-k+1} & \cdots & \binom{a_k+j-1}{j}
\end{vmatrix}
= \det\left( \binom{a_l+j-1}{j+l-i}_{1\leq i,l\leq k}\right),
\\
\lbrace a_1,\ldots,a_k \rbrace_{j} 
:=&
\begin{vmatrix}
\binom{a_k+j-1}{j} & \cdots & \binom{a_1+j+k-2}{j+k-1} \\
\vdots & \ddots & \vdots \\
\binom{a_k+j-k}{j-k+1} & \cdots & \binom{a_1+j-1}{j}
\end{vmatrix}
= \det\left( \binom{a_{k-l+1}+j+l-i-1}{j+l-i}_{1\leq i,l\leq k}\right)
\\
=&
\begin{vmatrix}
\binom{a_1+j-1}{j} & \cdots & \binom{a_1+j-1}{j+k-1} \\
\vdots & \ddots & \vdots \\
\binom{a_k+j-k}{j-k+1} & \cdots & \binom{a_k+j-k}{j}
\end{vmatrix}
=
 \det\left( \binom{a_{i}+j-i}{j+l-i}_{1\leq i,l\leq k}\right).
\end{align*}

\end{small}

The numbers $[a_1,\ldots,a_k]_{j}$ turn up in enumerative combinatorics, see for example~\cite{BK}, while the numbers $\lbrace a_1,\ldots,a_k \rbrace_{j}$ give the Hilbert polynomial of Schubert varieties due to the formula of Hodge and Pedoe, see Theorem III on page 387 of~\cite{HP} and also~\cite{Nan,Gh}. The equality of the two different formulae for $\lbrace a_1,\ldots,a_k \rbrace_{j}$ is given by Lemma 7 of~\cite{Gh}. In all cases, setting $k=1$ gives the multiset coefficients
$$
\multiset{a}{b}=\binom{a+b-1}{b}.
$$
If all $a_i$ are equal, which has been considered for the numbers $[a_1,\ldots,a_k]_{j}$ in~\cite{BK,MM,St}, then we have the following identity:

\begin{lemma}
Let $r,k-1 \in \ZZ_{\geq 0}$. It holds
$$
\lbrace\underbrace{ r+k-1,\ldots,r+k-1}_{k}  \rbrace_{j}
\ = \
[\underbrace{ r,\ldots,r }_{k}]_{j}
\ = \
[\underbrace{ j,\ldots,j }_{k}]_{r}.
$$
\end{lemma}

\begin{proof}
The second equality follows directly by transposing the matrix $M_k(r):=\binom{r+j-1}{j+l-i}_{1\leq i,l\leq k}$ and applying the binomial identity $\binom{a}{b}=\binom{a}{a-b}$.
We proof the first one. We see that the last rows of $M_{k}(r)$ and of $M'_{k}(r+k-1):=\binom{r+k+j-i-1}{j+l-i}_{1\leq i,l\leq k}$ are the same. Moreover, the lower arguments of the binomial coefficients in the $(i,l)$-th entry of  $M_{k}(r)$ and of $M'_{k}(r+k-1)$ are the same. The upper arguments in $M'_{k}(r+k-1)$ decrease by one if $i$ increases by one. Now recall the binomial identity
$$
\binom{a}{b}-\binom{a-1}{b-1}=\binom{a-1}{b}.
$$
By the elementary row operations of subtracting the second from the first, the third from the second and so on till we subtract the $k$-th from the $(k-1)$-th row and applying the above binomial identity, we decrease the upper arguments of the binomial coefficients in the first $k-1$ rows of $M'_{k}(r+k-1)$ each by one without changing the determinant. In particular, now the last \emph{two} rows of $M_{k}(r)$ and our new $M'_{k}(r+k-1)$ are the same. We do the same again, now only with the first $k-1$ rows and achieve that the last \emph{three} rows of the two matrices coincide. After doing this $k-1$ times, we have transferred $M'_{k}(r+k-1)$ to $M_{k}(r)$ without changing the determinant and the assertion is proven.
\end{proof}

\begin{definition}
Let  $k,r,j \in \ZZ_{\geq 1}$. Then we call the numbers
 $$
\mathcal{N}_k(r,j):=(\underbrace{r,\ldots,r}_{k})_{j},
\quad 
\mathfrak{N}_k(r,j):=[ \underbrace{r,\ldots,r}_{k}]_{j} = [\underbrace{j,\ldots,j}_{k}]_{r}
 $$
 the \emph{simple $k$-Narayana numbers} and the \emph{multiset $k$-Narayana numbers} respectively and  furthermore
 $$\mathcal{N}_{k,r}(t) := \sum_{j=0}^{r} \mathcal{N}_k(r,j)  t^j,
 \
 \mathfrak{N}_{k,r}(t) := \sum_{j=0}^{\infty} \mathfrak{N}_k(r,j)  t^j,
 \
 N_{k,r}(t) := \sum_{j=0}^{(r-1)(k-1)} N_k(r,j)  t^j
 $$
 the \emph{simple $k$-Narayana polynomial}, the \emph{$k$-Narayana series} and the \emph{$k$-Narayana polynomial} respectively.
\end{definition}

Apart from~\cite{BK}, they turn up in~\cite{MM} and~\cite{St}. In~\cite{BK} as well as in~\cite{St}, closed formulae for $\mathfrak{N}_k(r,j)$ are given, so that we get the following:

\begin{proposition}\label{prop:n-id}
Let $k,r,j \in \ZZ_{\geq 0}$ or  $k\geq 2$, $r >j \geq 0$ respectively. We have the identities
\begin{align*}
\mathfrak{N}_k(r,j) &= \prod\limits_{i=1}^{k} \multiset{j+i}{r-1} \multiset{i}{r-1}^{-1} =  \prod\limits_{i=1}^{k} \multiset{r-1+i}{j} \multiset{i}{j}^{-1},  \\
\mathcal{N}_k(r,j) &= \prod\limits_{i=1}^{k} \binom{r-1+i}{j} \multiset{i}{j}^{-1}
= \prod\limits_{i=0}^{k-1} \binom{r+i}{j} \binom{j+i}{j}^{-1}. \\
\end{align*}
\end{proposition}

\begin{proof}
The first identity is Theorem 3.3 in~\cite{BK}. The second follows by interchanging $j+1$ and $r$. The third by setting $r':=r+j-1$ in the second one and the last is a simple index shift.
\end{proof}

\begin{corollary}\label{cor:narid}
We have the following relations between $k$-Narayana numbers and polynomials:
\begin{align*}
\mathfrak{N}_k(r,j) &= \mathcal{N}_k(r+j-1,j) \\
                     &= \sum\limits_{l\geq 0} \binom{k(r-1)+j-l}{k(r-1)}N_k(r-1,l), \\
N_k(r,j) &= \sum\limits_{l=0}^{j}(-1)^{j-l} \binom{kr+1}{j-l} \mathcal{N}_k(r+l,l) \\
         &= \sum\limits_{l=0}^{j}(-1)^{j-l} \binom{kr+1}{j-l} \mathfrak{N}_k(r+1,l)\\
\mathfrak{N}_{k,r}(t) &= \frac{N_{k,r}(t)}{(1-t)^{k(r-1)+1}}
\end{align*}
\end{corollary}

\begin{proof}
The second equality is Proposition 4 of~\cite{Su2}, the last one is stated on page four of~\cite{Su2}, the others follow directly from the definitions and Proposition~\ref{prop:n-id}.
\end{proof}

Now we express the $k$-Narayana polynomials and series in terms of hypergeometric functions in order to find simplified formulae using a generalized Euler transform.

\begin{proposition}
Let $\mathcal{N}_{k,r}(t)$ be the simple $k$-Narayana polynomial and $\mathfrak{N}_{k,r}(t)$ the $k$-Narayana series. Let further 
${}_pF_q\left(
\begin{smallmatrix}
a_1,\ldots,a_p\\
b_1,\ldots,b_q
\end{smallmatrix};t\right)$
 be the generalized hypergeometric function. Then we have
\begin{align*}
\mathcal{N}_{k,r}(t) &= {}_{k}F_{k-1}\left(
\begin{matrix}
-r,\ldots,-r-k+1\\
2,\ldots,k
\end{matrix};(-1)^{k}t\right), \\
\mathfrak{N}_{k,r}(t) &= {}_{k}F_{k-1}\left(
\begin{matrix}
r,\ldots,r+k-1\\
2,\ldots,k
\end{matrix};t\right).
\end{align*}
\end{proposition}

\begin{proof}
We have
\begin{align*}
\mathfrak{N}_{k,r}(t)
&= 
\sum_{j \geq 0} \prod\limits_{i=1}^{k} \multiset{j+i}{r-1} \multiset{i}{r-1}^{-1} t^j 
\\
&=
\sum_{j \geq 0} \frac{(j+r-1)!\cdots(j+r+k-2)!}{(r-1)!^k j! \cdots (j+k-1)!} \frac{(r-1)!^k 1!\cdots k!}{(r-1)! \cdots (r+k-2)!} t^j
\\
&=
\sum_{j \geq 0} \frac{(r)_j\cdots(r+k-1)_j}{(2)_j \cdots (k)_j} \frac{t^j}{j!}
=
{}_{k}F_{k-1}\left(
\begin{matrix}
r,\ldots,r+k-1\\
2,\ldots,k
\end{matrix};t\right),
\\
\mathcal{N}_{k,r}(t)
&=
\sum_{j=0}^{r} \prod\limits_{i=0}^{k-1} \binom{r+i}{j} \binom{j+i}{j}^{-1}  t^j
\\
&=
\sum_{j=0}^{r} \frac{r!\cdots(r+k-1)!}{j!^k (r-j)! \cdots (r+k-1-j)!} \frac{j!^k 1!\cdots (k-1)!}{j! \cdots (j+k-1)!}  t^j
\\
&=
\sum_{j=0}^{r} \frac{(-r)_j\cdots(-r-k+1)_j}{(2)_j\cdots(k)_j}   \frac{(-1)^{kj}t^j}{j!}
\\
&=
{}_{k}F_{k-1}\left(
\begin{matrix}
-r,\ldots,-r-k+1\\
2,\ldots,k
\end{matrix};(-1)^{k}t\right).
\end{align*}

\end{proof}

\section{The Hilbert series of the Grassmannian}
Denote by $\GR(k,n)$ the complex Grassmannian of $k$-dimensional vector subspaces in $n$-dimensional vector space. The Hilbert series of its homogeneous coordinate ring under the Pl\"ucker embedding has been investigated for example in~\cite{Chip, GN, Ho1, Ho2, HP, Mu, SSW, Stu}. Mukai gives a closed formula for the Hilbert polynomial in terms of binomial coefficients in~\cite{Mu} while~\cite{GN} gives a closed formula for the generating rational function, both in the case $\GR(2,n)$.
In~\cite{Ho1}, Hodge conjectured a closed formula for the Hilbert polynomial in the general case $\GR(k,n)$, which was proved by Littlewood in ~\cite{Li}. This formula is the following:
$$
d_{k,n}(j)=\frac{(n+j)!\cdots (n+j-k)! }{j! \cdots (k+j)! }\frac{1! \cdots k!}{(n-k)! \cdots n!}.
$$
The homogeneous coordinate ring of $\GR(k,n)$ equals $\CC\left[V^{n}\right]^{\SL_k}$, where the action of $\SL_k$ on $V^{n}$ is induced by multiplication from the left on the $k$-dimensional vector space $V$, see for example~\cite{LaBr}, Theorem 9.3.6. So the Hilbert series of $\GR(k,n)$ and $\CC\left[V^{n}\right]^{\SL_k}$ coincide.

The following theorem gives a formula for the Hilbert series in terms of binomial coefficients and an expression of the Hilbert series as a rational function.

\begin{theorem}\label{th:hilbgrass}
The $h$-polynomial of the complex Grassmannian $\GR(k,n)$ is the $k$-Narayana polynomial $N_{k,n-k+1}$ and its Hilbert series is the $k$-Narayana series $\mathfrak{N}_{k,n-k+1}$. That is to say:
$$
H(\GR(k,n))  = \mathfrak{N}_{k,n-k+1}(t^k) =  \frac{N_{k,n-k+1}(t)}{(1-t^k)^{k(n-k)+1}}.
$$
\end{theorem}

\begin{lemma}\label{le:van}
For $k\geq 2$,  the following equality holds in $\CC\left( z_1,\ldots,z_k \right)$:
$$
\prod\limits_{1 \leq i < j \leq k} \left(1-\frac{z_i}{z_j}\right)
=
\left|
\left(z_i^{j-i}\right)_{1\leq i,j \leq k }
\right|
$$
\end{lemma}

\begin{proof}
The Vandermonde matrix 
$$
\left(z_i^{j-1}\right)_{1\leq i,j \leq k }
$$ 
has the determinant
$$
\left|
\left(z_i^{j-1}\right)_{1\leq i,j \leq k }
\right|
=
\prod\limits_{1 \leq i < j \leq k} \left(z_j-z_i\right).
$$
So we have
\begin{align*}
\left|
\left(z_i^{j-i}\right)_{1\leq i,j \leq k }
\right|
=&
\left(\prod\limits_{i=1}^{k}z_i^{1-i} \right) 
\left|
\left(z_i^{j-1}\right)_{1\leq i,j \leq k }
\right| \\
=&
\left(\prod\limits_{i=1}^{k}z_i^{1-i} \right)\left(\prod\limits_{1 \leq i < j \leq k} \left(z_j-z_i\right)\right) \\
=& \prod\limits_{1 \leq i < j \leq k} \left(1-\frac{z_i}{z_j}\right).
\end{align*}
\end{proof}

\begin{proof}[Proof of Theorem~\ref{th:hilbgrass}]
Let $n\geq k$.
We begin with what Mukai~\cite{Mu} calls the \emph{$q$-Hilbert series} of the action of $SL_k$ on $V^n$ induced by multiplication from the left on the $k$-dimensional vector space $V$. 

Let $T=SL_k \cap {\rm Diag}(k) \cong \left(\CC^*\right)^{k-1}$ be the standard maximal torus of $SL_k$ and $X(T) \cong \ZZ^{k-1}$ its character lattice with standard basis $e_1,\ldots,e_{k-1}$. Let 
$$
\lambda_i := \sum\limits_{j=1}^{i} e_j \in X(T), \quad i=1,\ldots,k-1
$$
be the fundamental weights and 
\begin{align*}
\alpha_1 &:= e_1-e_{2}  = -\lambda_{2}+2\lambda_1 &&\in X(T) \\ 
\alpha_i &:= e_i-e_{i+1}  = -\lambda_{i+1}+2\lambda_i-\lambda_{i-1} &&\in X(T), \quad i=2,\ldots,k-2 \\ 
\alpha_{k-1} &:= 2e_{k-1} + \sum\limits_{j=1}^{k-2} e_j = 2\lambda_{k-1}-\lambda_{k-2} &&\in X(T)
\end{align*}
the simple roots. Define $q_i:=\chi_{\lambda_i}$ for $i=1,\ldots,k-1$. So the set  of positive roots is
\begin{align*}
\Phi^+ 
&= \left\{ \left. \sum\limits_{i=a}^b \alpha_i \right| 1 \leq a\leq b \leq k-1\right\} \\ 
&= \left\{ \left. \frac{q_a q_b}{q_{a-1} q_{b+1}} \right| 1 \leq a\leq b \leq k-1, q_0=q_k=1\right\}.
\end{align*} 
 On $V$, the action of $T$ is given by the matrix
$$
{\rm Diag}\left( q_1, q_2q_1^{-1},q_3q_2^{-1},\ldots, q_{k-1}q_{k-2}^{-1}, q_{k-1}^{-1} \right).
$$
Now with~\cite{KD}, Remark 4.6.10, the Hilbert series of $\CC[V^n]^{\SL_k}$ is the coefficient of $q_1^0\cdots q_{k-1}^0$ when we set $q_0=q_k=1$ in the $q$-Hilbert series:
\begin{align*}
H_q(\CC[V^n]^{\SL_k})
&= 
\frac{\prod\limits_{1 \leq a\leq b \leq k-1} \left(1-\frac{q_a q_b}{q_{a-1} q_{b+1}}\right)}{\prod\limits_{j=1}^k \left(1 - t \frac{q_j}{q_{j-1}} \right)^{n}} \\
&= \frac{
\left|
\left(\left(\frac{q_i}{q_{i-1}}\right)^{j-i}\right)_{1\leq i,j \leq k }
\right|
}{\prod\limits_{j=1}^k \left(1 - t \frac{q_j}{q_{j-1}}  \right)^{n}}\\
&=
\left|
\left(\left(\frac{q_i}{q_{i-1}}\right)^{j-i}\right)_{1\leq i,j \leq k }
\right|
\sum_{i_1,\ldots,i_k=0}^\infty \prod_{j=1}^k \multiset{n}{i_j}\left(t\frac{q_j}{q_{j-1}}\right)^{i_j}  
\\
&=
\sum_{\sigma \in S_k} {\rm sgn}(\sigma) 
\sum_{i_1,\ldots,i_k=0}^\infty \prod_{j=1}^k \multiset{n}{i_j} t^{i_j} \left(\frac{q_j}{q_{j-1}}\right)^{\sigma(j)+i_j-j}
\end{align*}
For the second equality, we used Lemma~\ref{le:van} with $z_i=q_i/q_{i-1}$. Now for the coefficient of $q_1^0\cdots q_{k-1}^0$, for each $\sigma \in S_k$, all the exponents $\sigma(j)+i_j-j$ must be the same, so that we get:
\begin{align*}
H(\CC[V^n]^{\SL_k})
&=
\sum_{\sigma \in S_k} {\rm sgn}(\sigma) 
\sum_{\tiny
\begin{array}{c}
i_1,\ldots,i_k \geq 0 \\
~\sigma(1)+i_1-1\\
=\ldots=\\
\sigma(k)+i_k-k
\end{array}
} \prod_{j=1}^k \multiset{n}{i_j} t^{i_j} 
\\
&= 
\sum_{\sigma \in S_k} {\rm sgn}(\sigma) 
\sum_{\tiny
\begin{array}{c}
i_1 \geq \sigma(1)-1\\
,\ldots, \\
 i_k \geq \sigma(k)-k \\
i_1=\ldots=i_k
\end{array}
} \prod_{j=1}^k \multiset{n}{i_j-\sigma(j)+j} t^{i_j-\sigma(j)+j} 
\\
&=
\sum_{\sigma \in S_k} {\rm sgn}(\sigma) 
\sum_{l=0} \prod_{j=1}^k \multiset{n}{l-\sigma(j)+j} t^{l-\sigma(j)+j} 
\\
&=
\sum_{l=0}^{\infty} t^{kl} \sum_{\sigma \in S_k} {\rm sgn}(\sigma) 
\multiset{n}{l-\sigma(j)+j} 
=
\sum_{l=0}^{\infty}  \left| \multiset{n}{l-i+j}_{1\leq i,j \leq k} \right|  t^{kl} \\
&=
\sum_{l=0}^{\infty} \{n,\ldots,n\}_l  t^{kl} 
=
\sum_{l=0}^{\infty} [n-k+1,\ldots,n-k+1]_l  t^{kl} \\
&= \mathfrak{N}_{k,n-k+1}(t^k) = \frac{N_{k,n-k+1}(t)}{(1-t^k)^{k(n-k)+1}}.
\end{align*}
The second and third equalities hold because for each $\sigma$, if there is a $1\leq j \leq k$ with $\sigma(j)-j >0$, there must be a $j'$ with $\sigma(j')-j' >0$. Furthermore, since $k \leq n$, for all $1\leq j \leq k$ we have $-n < \sigma(j)-j $. Thus the additional as well as the removed summands all are zero and the second and third equality hold.
The last identity directly follows from Corollary~\ref{cor:narid}.

\end{proof}

\begin{remark}
Of course the quotient of $V^n$ by $SL_k$ is also defined for $n<k$, but the only invariant functions in these cases are the constant ones, so that $H(\CC[V^n]^{\SL_k})=1$ holds here.

The equality of our formula for the Hilbert polynomial and that of Hodge is immediate. There is another way to compute it using the Borel-Weil-Bott Theorem, see Section 5.3 of~\cite{Chip}.
For the connection to the approach~\cite{Stu} of Sturmfels, see also the paper~\cite{SSW}.
\end{remark}

As we already stated in Section~\ref{sec:nar}, the numbers $\lbrace a_1,\ldots,a_k \rbrace_{j}$ with $1\leq a_1 < \ldots < a_k \leq n$ give the Hilbert polynomials of  Schubert varieties 
$$
X(a_1,\ldots,a_k):=\lbrace W \in \GR(k,n) | \dim(W \cap <e_1,\ldots,e_{a_i}>) \geq i,i=1,\ldots,k\rbrace.
$$
From this, Nanduri in~\cite{Nan} deduced a closed form for the $h$-polynomial of $X(a_1,\ldots,a_k)$. We give a slightly simpler formula in the following, which reduces to the formula of Sulanke for the $k$-Narayana numbers when we set $a_1=\ldots=a_k=n$.

\begin{proposition}
The $i$-th coefficient $h_i$ of the $h$-polynomial of the Schubert variety $X(a_1,\ldots,a_k)$ of dimension $d$ is given by
$$
h_i= \sum_{l=0}^i (-1)^l \lbrace a_1,\ldots,a_k \rbrace_{i-l} \binom{d}{l}.
$$
\end{proposition}

\begin{proof}
According to the proof of Proposition 2.9 of~\cite{Nan}, the coefficient $h_i$ is given by
$$
h_i=\sum_{l=0}^i (-1)^l \lbrace a_1,\ldots,a_k \rbrace_{i-l} \binom{d-1}{l}-\sum_{l=0}^{i-1} (-1)^l \lbrace a_1,\ldots,a_k \rbrace_{i-l-1} \binom{d-1}{l}.
$$
We have
\begin{align*}
h_i
&=
\sum_{l=0}^i (-1)^l \lbrace a_1,\ldots,a_k \rbrace_{i-l} \binom{d-1}{l}-\sum_{l=1}^{i} (-1)^{l-1} \lbrace a_1,\ldots,a_k \rbrace_{i-l} \binom{d-1}{l-1} \\
&= 
\sum_{l=0}^i (-1)^l \lbrace a_1,\ldots,a_k \rbrace_{i-l} \binom{d-1}{l}+\sum_{l=0}^{i} (-1)^l \lbrace a_1,\ldots,a_k \rbrace_{i-l} \binom{d-1}{l-1} \\
&= 
\sum_{l=0}^i (-1)^l \lbrace a_1,\ldots,a_k \rbrace_{i-l} \binom{d}{l}.
\end{align*}
\end{proof}

\section{Generalized hypergeometric Euler transform}
For integers $p,q\geq 0$ and $a_1,\ldots,a_p,b_1,\ldots,b_q \in \CC$, where no $b_i$ is an integer smaller than one, recall the generalized hypergeometric function
$$
{}_{p}F_{q}\left(
\begin{matrix}
a_1,\ldots,a_p\\
b_1,\ldots,b_q
\end{matrix};t\right):= \sum\limits_{k=0}^{\infty} \frac{\left(a_1\right)_k \cdots \left(a_p\right)_k}{\left(b_1\right)_k \cdots \left(b_q\right)_k} \frac{t^k}{k!},
$$
where $\left(a\right)_k$ is the Pochhammer symbol standing for the rising factorial $a(a+1)\cdots(a+k-1)$ with $(a)_0=1$. Consider the Euler transformation for the ordinary hypergeometric function ${}_2F_1$:
$$
{}_{2}F_{1}\left(
\begin{matrix}
a,b\\
c
\end{matrix};t\right) = (1-t)^{c-a-b} {}_{2}F_{1}\left(
\begin{matrix}
c-a,c-b\\
c
\end{matrix};t\right).
$$
This identity in particular shows that the $2$-Narayana polynomial is a hypergeometric function, i.e. the identity
$$
\sum\limits_{j=0}^{k}(-1)^{k-j} \binom{2r+1}{k-j} \prod\limits_{i=0}^{1}\binom{r+i+j}{r}\binom{r+i}{r}^{-1}
=
\frac{1}{r}\binom{r}{k}\binom{r}{k+1}.
$$
The hypergeometric Euler transform has been generalized to ${}_pF_q$ in different ways, see~\cite{Ma}, Theorem 2.1 and~\cite{MP1}, Theorem 3 as well as~\cite{MP2}, Theorem 4. 
We will use Theorem 4 of~\cite{MP2} to express $k$-Narayana numbers as products for $k\geq 3$ as well. This leads to a representation of the $h$-polynomials of Grassmannians as hypergeometric functions for $k\geq 3$. 

In order to do this, we first state the generalized hypergeometric Euler transform. Denote by 
$\smallstir{x}{y}
$ the Stirling numbers of the second kind. Consider for $r \in \ZZ_{\geq 1}$ numbers $m_1,\ldots,m_r \in \ZZ_{\geq 1}$ and set $m=\sum m_r$. Furthermore let $a,b,c,f_1,\ldots,f_r \in \CC$ meet the requirements
$$
(c-a-m)_m, (c-b-m)_m, (1+a+b-c)_m \neq 0.
$$
Now for $j,k=0,\ldots,m$ denote by $\sigma_{j}$  the coefficient of $x^j$ in the polynomial $(f_1+x)_{m_1}\cdots (f_r+x)_{m_r}$ and set
\begin{align*}
A_k &:=\sum_{j=k}^m \stir{j}{k} \sigma_{j}, \\
G_k(t) &:={}_3F_2\left(
\begin{matrix}
k-m,k-t,1-c-t\\
1+b+k-c-t,1+a+k-c-t
\end{matrix};1\right), \\
Q(t) &:= \sum_{k=0}^m (-1)^kA_k(a)_k(b)_k(t)_k(c-a-m-t)_{m-k}(c-b-m-t)_{m-k}G_k(-t).
\end{align*}
With these definitions, we get the following:

\begin{theorem}[Generalized hypergeometric Euler transform,~\cite{MP2}, Th. 4]\label{th:geneul}
Let $\eta_1,\ldots,\eta_m$ be the nonvanishing zeros of the polynomial $Q(t)$. Then for $|t|<1$, the following equality holds true:
\begin{align*}
&{}_{r+2}F_{r+1}\left(
\begin{matrix}
a,b,f_1+m_1,\ldots,f_r+m_r\\
c,f_1,\ldots,f_r
\end{matrix};t\right)  \\
=& 
(1-t)^{c-a-b-m}
{}_{m+2}F_{m+1}\left(
\begin{matrix}
c-a-m,c-b-m,\eta_1+1,\ldots,\eta_m+1\\
c,\eta_1,\ldots,\eta_m
\end{matrix};t\right).
\end{align*}
\end{theorem}

Now consider the case of the $k$-Narayana series
$$
\mathfrak{N}_{k,r}(t) = {}_{k}F_{k-1}\left(
\begin{matrix}
r,\ldots,r+k-1\\
2,\ldots,k
\end{matrix};t\right)
$$
with $k\geq 3$. 
Letting $c:=2, a:=r+k-2, b:=r+k-1$ and for $i=3,\ldots,k$ furthermore $f_i:=i, m_i:=r-3$, so that $m= \sum m_i = (k-2)(r-3)$, we have that none of
$$
(4-r-k-(k-2)(r-3))_{(k-2)(r-3)},
\quad (3-r-k-(k-2)(r-3))_{(k-2)(r-3)}
$$
$$
 (2(r+k)-4)_{(k-2)(r-3)} 
$$
equals zero. We thus get:

\begin{theorem}\label{th:multform}
The $k$-Narayana polynomial can be expressed  as the hypergeometric function
$$
N_{k,r} =
 {}_{m+2}F_{m+1}
\left(
\begin{matrix}
-(k+1)r-2k-3,-(k+1)r-2k-2,\eta_1+1,\ldots,\eta_{m}+1\\
2,\eta_1,\ldots,\eta_{m}
\end{matrix};t\right) 
$$
with $m=(k-2)(r-3)$. Moreover, for the $k$-Narayana numbers $N_k(r,j)$, we have the product formula
$$
N_k(r,j)= \binom{(k+1)r+2k+3}{j}\binom{(k+1)r+2k+2}{j} \prod_{i=1}^{(k-2)(r-3)}\frac{\eta_i+j}{\eta_i},
$$
where the $\eta_i$ are the zeros in $t$ of the polynomial
{\small
\begin{align*}
\sum_{l=0}^{(k-2)(r-3)}
& (-1)^l A_l  (-(k+1)r-2k-3-t)_{(k-2)(r-3)-l}(-(k+1)r-2k-t)_{(k-2)(r-3)-l} \\
&
(r+k-2)_l(r+k-1)_l(t)_l~{}_3F_2\left(
\begin{matrix}
l-(k-2)(r-3),l+t,t-1\\
r+k-2+l+t,r+k-3+l+t
\end{matrix};1\right).
\end{align*}}
\end{theorem}

\begin{proof}
Letting $c=2, a=r+k-2, b=r+k-1$ and for $i=3,\ldots,k$ furthermore $f_i=i, m_i=r-3$ as stated above, the requirements of Theorem~\ref{th:geneul} are fulfilled and applying it we get
\begin{align*}
\mathfrak{N}_{k,r}(t) 
&=
{}_{k}F_{k-1}\left(
\begin{matrix}
r+k-2,r+k-1,r,\ldots,r+k-3\\
2,3,\ldots,k
\end{matrix};t\right) \\
&=\frac{ {}_{m+2}F_{m+1}
\left(
\begin{smallmatrix}
-(k+1)r-2k-3,-(k+1)r-2k-2,\eta_1+1,\ldots,\eta_{m}+1\\
2,\eta_1,\ldots,\eta_{m}
\end{smallmatrix};t\right)}{(1-t)^{(r-1)k+1}}
\end{align*}
with $m=(k-2)(r-3)$. So with Corollary~\ref{cor:narid}, we have
$$
N_{k,r} =
 {}_{m+2}F_{m+1}
\left(
\begin{matrix}
-(k+1)r-2k-3,-(k+1)r-2k-2,\eta_1+1,\ldots,\eta_{m}+1\\
2,\eta_1,\ldots,\eta_{m}
\end{matrix};t\right) 
$$
and thus
\begin{align*}
N_k(r,j)
&=
\frac{(-(k+1)r-2k-3)_j,(-(k+1)r-2k-2)_j (\eta_1+1)_j \cdots (\eta_m+1)_j}{(2)_j  (\eta_1)_j \cdots (\eta_m)_j j!} \\
&= \binom{(k+1)r+2k+3}{j}\binom{(k+1)r+2k+2}{j} \prod_{i=1}^{(k-2)(r-3)}\frac{\eta_i+j}{\eta_i}.
\end{align*}
\end{proof}

\begin{corollary}
The $h$-polynomial of the Grassmannian $\GR(k,n)$ is given by
{
$$
 {}_{m+2}F_{m+1}
\left(
\begin{matrix}
k(k-n)-n-3,k(k-n)-n-2,\eta_1+1,\ldots,\eta_{m}+1\\
2,\eta_1,\ldots,\eta_{m}
\end{matrix};t\right)
$$}
with $m=k(n-k)-2n+6$.
\end{corollary}

\begin{remark}\label{rem:red}
Observe the case $k\geq r$. Here we have the reduction
\begin{align*}
\mathfrak{N}_{k,r}(t) 
&=
 {}_{k}F_{k-1}\left(
\begin{matrix}
r,\ldots,k,k+1,\ldots,r+k-1\\
2,\ldots,r-1,r,\ldots,k
\end{matrix};t\right) \\
 &=
 {}_{r-1}F_{r-2}\left(
\begin{matrix}
k+1,\ldots,(k+1)+(r-1)-1\\
2,\ldots,r-1
\end{matrix};t\right) 
 = \mathfrak{N}_{r-1,k+1}(t),
\end{align*}
from what follows $N_{k,r}(t)=N_{r-1,k+1}(t)$.
In this case, it makes sense to first use this reduction and afterwards apply the formula of Theorem~\ref{th:multform} to get a simpler formula.
\end{remark}

\begin{example}
We implemented our formula for the $k$-Narayana numbers from Theorem~\ref{th:multform} in Maple. The following table shows some results for $k=3$ and values of $r$ that cannot be reduced to simpler forms with Remark~\ref{rem:red}, i.e. for $r \geq 4$.

\renewcommand{\arraystretch}{1.8} 

\begin{longtable}{c|c}
 $r$ & $N_k(r,j)$ \\
\hline
 $4$ & $\binom{25}{j}\binom{24}{j} (1+4j)$ \\
 \hline
 $5$ & $\frac{1}{105}\binom{29}{j}\binom{28}{j} \left(-3-\sqrt{114}+j\right)\left(-3+\sqrt{114}+j\right)$ \\
 \hline
 $6$ & $\frac{1}{63}\binom{33}{j}\binom{32}{j} \left(-4-\sqrt{79}+j\right)\left(-4+\sqrt{79}+j\right)$ \\
 \hline
 $7$ &  $\frac{1}{495}\binom{37}{j}\binom{36}{j} \left(-5-2\sqrt{130}+j\right)\left(-5+2\sqrt{130}+j\right)\left(1+\frac{j}{4}\right)\left(1-\frac{j}{14}\right)$ \\
 \hline
 $8$ &  $\begin{smallmatrix}
 \frac{5}{61776}\binom{41}{j}\binom{40}{j} \left(-6-\frac{1}{10}\sqrt{17450 - 10 \sqrt{682705}}+j\right)\left(-6-\frac{1}{10}\sqrt{17450 + 10 \sqrt{682705}}+j\right) \\
 \cdot\left(-6+\frac{1}{10}\sqrt{17450 - 10 \sqrt{682705}}+j\right)\left(-6+\frac{1}{10}\sqrt{17450 + 10 \sqrt{682705}}+j\right)
\end{smallmatrix}$ 
\end{longtable}
\end{example}

\end{document}